\documentclass[12pt]{amsart}
\usepackage{amssymb}
\usepackage{latexsym}
\usepackage{epsf}
\usepackage{amssymb, epic,eepic,epsfig,amsbsy,amsmath,amscd}
\newtheorem{thm}{Theorem}[section]
\newtheorem{Thm}{Theorem}
\newtheorem{pro}[thm]{Proposition}
\newtheorem{lem}[thm]{Lemma}

\newtheorem{Cor}[Thm]{Corollary}

\newcommand{\qbinom}[2]{\text{$\left[\begin{array}{c}#1\\ #2\end{array}
\right]$}}

\theoremstyle{definition}

\def\1{{\rm1\mathchoice{\kern-0.25em}{\kern-0.25em}
        {\kern-0.2em}{\kern-0.2em}I}}

\newcommand{\lmn}[1]{\vadjust{\setbox1=\vtop{\hsize 25mm
\parindent=0pt\baselineskip=9pt
\rightskip=4mm plus 4mm#1}
\hbox{\kern-26mm\smash{\raise .5ex\box1}}}}

\newcommand{\nc}{\newcommand}
\def\be#1\ee{\begin{equation}#1\end{equation}}
\nc{\bc}{\begin{center}} \nc{\ec}{\end{center}} \nc{\bb}{\mathbb}
\nc{\cal}{\mathcal} \nc{\frk}{\mathfrak} \nc{\N}{{\mathsf N}}
\nc{\K}{{\mathsf K}} \nc{\fk}{\mathbf{k}} \nc{\fn}{\mathbf{n}}
\nc{\fb}{\mathbf{b}}  \nc{\e}{\varepsilon} \nc{\ev}{{\rm{ev}}}

\theoremstyle{remark}

\def\Z{{\mathbb Z}}
\def\Q{{\mathbb Q}}
\def\N{{\mathbb N}}

\def\v8{\vskip 8pt}

\def\la{\langle}
\def\ra{\rangle}
\def\l{\lambda}
\def\n{\nu}
\def\g{\gamma}
\def\m{\mu}

\def\ZZ{\widehat{\Z[v]}_2}
\def\ZZZ{\widehat{\Z[v]}_s}
\def\Habiro{\widehat{\Z[q]}}

\def\R{\mathcal R}
\def\BN{\mathbb N}
\def\BZ{\mathbb Z}

\def\BQ{\mathbb Q}
\def\BR{\mathbb R}

\def\cT{\mathcal T}
\def\cM{\mathcal M}

\begin{document}

\title[Unified quantum invariants and their refinements ]{Unified
quantum invariants and their refinements\\[.2cm]
for homology  3--spheres with 2--torsion}
\author{Anna Beliakova}
\address{Institut f\"ur Mathematik, Universit\"at Z\"urich,
 Winterthurerstrasse 190,
CH-8057 Z\"urich, Switzerland}
\email{anna@math.unizh.ch}
\author{Christian Blanchet}
\address{L.M.A.M., Universit\'e de Bretagne-Sud,
Centre de Recherche Tohannic, BP 573, F-56017 Vannes, France  }
\email{Christian.Blanchet@univ-ubs.fr}
\author{Thang T.Q. L\^e}
\address{School of Mathematics, Georgia Institute of Technology, Atlanta,
GA 30332-0160, USA} \email{letu@math.gatech.edu}
\date{February 2007}

\begin{abstract}
For every  rational homology $3$--sphere with $H_1(M,\Z)=(\Z/2\Z)^n$
we construct a unified invariant (which takes values in a certain
cyclotomic completion of a polynomial ring), such that
the evaluation of this invariant at any  odd  root of unity
provides the SO(3)  Witten--Reshetikhin--Turaev invariant at this root
and  at any even root of unity
 the SU(2) quantum invariant.
Moreover, this unified invariant splits into a sum of the refined unified
invariants dominating spin and cohomological refinements of
quantum SU(2) invariants.
New results on the Ohtsuki series and the integrality  of quantum invariants
are the main applications of our construction.

\end{abstract}

\maketitle
\section*{Introduction}

The Witten--Reshetikhin--Turaev (WRT) invariants of 3--manifolds,
also known as
  quantum invariants,
  are defined only when the quantum parameter $q$ is a
certain root of unity. In \cite{Ha},
Habiro proposed a construction of a
unified  invariant of integral homology $3$--spheres,
 dominating all  quantum $SU(2)$ invariants.

The unified invariant  is an element of the Habiro
ring
$$ \Habiro:=\lim_{\overleftarrow{\hspace{2mm}n\hspace{2mm}}}
\frac{ \Z[q]}{
(1-q)(1-q^2)...(1-q^n)}\, .$$
Every element $f\in \Habiro$ can be written as an infinite sum
$$ f(q)= \sum_{k\ge 0} f_k(q)\, (1-q)(1-q^2)...(1-q^k),$$
with $f_k(q)\in \Z[q]$. If $\xi$ is a root of unity, then $f(\xi)$ is
well--defined, since the summands become zero if $k$ is bigger than the
order of $\xi$. The Habiro ring has remarkable properties and is very
suitable for the study of quantum invariants. The result of Habiro
 mentioned above is

\begin{Thm}{\rm (Habiro)}
For every   integral homology $3$--sphere $M$, there exists an
invariant $I_M(q)\in \Habiro$, such that if $\xi$ is a root of
unity, then $I_M(\xi)$ is the WRT invariant at $\xi$.
\end{Thm}

Let us mention the most important consequences of the Habiro's
construction. First of all, each product
$$(q;q)_n:=(1-q)(1-q^2)\dots (1-q^n)$$
is divisible by $(1-q)^n$, hence it is easy to expand every $f(q)\in
\Habiro$ into formal power series in $(q-1)$, denoted by $T(f)$ and
called the Taylor series of $f(q)$ at $q=1$. One important property
of $\Habiro$ is that $f\in \Habiro$ is uniquely determined by its
Taylor series. In other words, the map $T: \Habiro \to \Z[[q-1]]$
is injective. In particular, $\Habiro$ is an integral domain.
Moreover, every $f\in \Habiro$ is
determined by the values of $f$ at any infinite set of roots of
unity of prime power  order. From the existence of $I_M$ one can
derive the following consequences  for integral homology spheres:
\begin{itemize}
\item
The quantum invariants   at all roots of unity are algebraic
integers.

\item The quantum invariants  at any infinite set of roots
 of unity of
prime power  order determine the whole set of quantum invariants.

\item
The  Ohtsuki series (see \cite{Oh, Le2}) have integer coefficients and
 determines the whole set of quantum invariants.

 \item
 The Le--Murakami--Ohtsuki invariant (see \cite{LMO}) totally
 determines the quantum invariants.
\end{itemize}

The integrality of quantum invariants was established earlier only
at roots of unity of prime order (see \cite{MR,Le1}). The integrality
of the Ohtsuki series  was proven by Rozansky, using
quite a different method.

 In \cite{Le;last}, the third author extended  Habiro's results
to rational homology 3--spheres. More precisely,
for a 3--manifold $M$ with $|H_1(M,\Z)|=a$, he constructed
a unified invariant $I_M$ dominating quantum SO(3) invariants
of $M$ at roots unity of order odd and co--prime to $a$.
The Habiro ring was modified by inverting $a$ and all cyclotomic
polynomials not co--prime to $a$.

More precisely,
 for $t:=q^{1/a}$, let $\R_{a,k}$ be a subring
of $\Q(t)$ generated  over $\Z[t^{\pm 1}]$
 by
  $\frac{(t;t)_{k}}{(q;q)_{k}}$. Note that,
 $$ \R_{a,1} \subset \R_{a,2} \subset \dots \subset \R_a \,$$
where $\R_a= \cup_{k=1}^\infty \R_{a,k}$.
Let
$$\widehat \R_a := \lim_{\overleftarrow{\hspace{2mm}n\hspace{2mm}}}
\frac{\R_a}{(q;q)_n}$$ be its cyclotomic completion.
Every element $f\in \widehat\R_a$ has the following presentation:
\be \label{BL}
f=\sum^\infty_{n=0} f_n(t)\; \frac{(1-q^{n+1})(1-q^{n+2})\dots(1-q^{2n+1})}{1-q}\,  ,  \ee
 where $f_n(t)\in \R_a$.
 It was shown in \cite{Le;last} that $I_M\in \widehat\R_a$.

Let $\Gamma_a$ be the set of all elements of $\widehat\R_a$
that have presentation (\ref{BL}) such that $f_n(t)\in \R_{a,2n+1}$.
In \cite{BL}, the first and the third authors  proved that 
$\Gamma_a$ is smaller then $\widehat \R_a$ and that
$I_M\in \Gamma_a$, i.e. the unified invariant 
 has even stronger integrality.


The  
results in \cite{Le;last} and \cite{BL}
 concern only the $SO(3)$ invariant,
for which the order of quantum parameter must be odd. In this paper
we mainly address the case of $SU(2)$ when the order is even. We
construct a unified invariant dominating quantum SO(3) and SU(2)
invariants of rational  homology 3--spheres with 2--torsion only. We show that
our unified invariant splits into a sum of refined unified
invariants dominating spin and cohomological refinements of quantum
SU(2) invariants. New results on integrality of quantum invariants
and the Ohtsuki series are the main applications of our
construction.

\subsection{Results}  Let $\cM_n$ be the set of all
oriented closed compact 3--manifolds $M$ with $H_1(M,\Z)=(\Z/2\Z)^n$. When
$M$ is not an integral homology 3--sphere, the WRT $SU(2)$ invariant,
denoted here by $\tau_M(\xi)$, depends on a 4--th root of $\xi$,
although we will not signify this in our notation. 

\def\cM{\mathcal M}
\def\RP{{\mathbb R}P}

Suppose $M \in \cM_n$. If $\xi$ is an odd root of unity, then
$\tau_M(\xi)=0$, but $\tau^{SO(3)}_M(\xi)\neq 0$. In this case, we
choose $\zeta$  to be the square root of $\xi$, which has the {\it
same order} as $\xi$, and  put
$$\tau'_M(\xi)= \tau^{SO(3)}_M(\xi)/(\tau^{SO(3)}_{\RP^3}(\xi))^n,$$
where $\RP^3$ is the projective space.
\newcommand{\G}{\Gamma}

If  the order of $\xi$ is  {\it even} (then the order of $\zeta$ is
divisible by 4),  define

$$\tau'_M (\xi)= \tau_M(\xi)/(\tau_{\RP^3}(\xi))^n\, .$$

Note that   the quantum invariant $\tau'_M(\xi)$
depends only on a square root of $\xi$.

For $q=v^2$, we define
$\G:=\Z[1/2][v]$.
Let $S=\{n\in \N|n\not\equiv 2 \pmod 4\}$.
The
cyclotomic completion $\G^S$
of the polynomial ring $\G$ with respect to
$S$ was defined by Habiro in \cite{Ha1} (the definition is recalled in
Section \ref{cyc}). 




Suppose  $\xi$ is a root of unity.  Fix a square root $\zeta$ of  $\xi$,
such that if the order of $\xi$ is odd, then the order of $\zeta$ is
the same as that of $\xi$. Then one can evaluate every element $f\in
 \G^S$ by replacing $v$ with $\zeta$; the result is a
complex number, denoted by $\ev_\xi(f)$.

Our first main result is
\begin{Thm}\label{Two}
For every  closed oriented manifold $M\in \cM_n$, there exists a unique
invariant $I_M \in  {\G^S }$, such that for every root $\xi$
of unity and a choice of its square root $\zeta$ as above, one has
 $\ev_\xi(I_M)=\tau'_M(\xi)$.
\end{Thm}


The ring ${\G^S }$ is a smaller  than $\G_2$  used in \cite{BL},
because of
 the factors $(1+q)(1+q^2)...(1+q^n)$ in the completion,
which are responsible for the spin and cohomological refinements.
 Hence, when restricted to SO(3) invariants, the integrality
stated in Theorem \ref{Two} is stronger than that in 
\cite{BL}.

\begin{Cor}\label{cor-gen}
For $M\in \cM_n$, and the quantum invariants
$\tau'_M$, the following statements hold.

(a)
The  quantum invariants belong to $\Z[1/2][\zeta]$.

(b)
The quantum invariants
are determined by their values at roots of unity $\zeta$ with
${\rm ord(\zeta)}=\{2^k p^e\, | k\in \N, e\in \BN\}$ 
for any odd prime $p\neq 1$.

\end{Cor}

Clearly, part (b)  holds also for  $\tau_M$,
however for the part (a) to be true, one may need to multiply $\tau_M$ by
$\zeta^{1/2}$.
\v8

\noindent
{\bf Open problem} We do not know whether these  invariants
are determined by the
 Le--Murakami--Ohtsuki invariant or not.

\subsubsection{Spin  and cohomological refinements} Suppose now the
order of $\xi$ is even, i.e. the order of $\zeta$ is divisible by 4.
 There are refined quantum invariants
$\tau_{M,\sigma}(\xi)$, defined in [1],\cite{KM}, where $\sigma$ is
a spin structure or a cohomological class in $H^1(M,\Z/2\Z)$,
depending on whether the order of $\zeta$ is congruent to $0 \pmod
8$ or $4 \pmod 8$. 

We will renormalize $\tau_{M,\sigma}$ by dividing
by the non--refined invariant of the projective space, i.e.
$\tau'_{M,\sigma}(\xi) : =\tau_{M,\sigma}(\xi)/
(\tau_{\RP^3}(\xi))^n$.
 Then we have
$\tau'_M(\xi)=\sum_\sigma \tau'_{M,\sigma}(\xi)$.

For   $T=\{n\in \N\,|\, n\equiv 0\mod 8\}$,
let $\G^T$ be the cyclotomic completion with respect to $T$ as
defined in
Section \ref{cyc}. If $\zeta$ is a root of unity with order in $T$,
then we can evaluate any element $f\in \G^{T}$ by replacing $v$
with $\zeta$; the result is a complex number denoted by
$\ev_\xi(f)$.

Our next result is
\begin{Thm}\label{main_spin}
For a manifold $M\in \cM_n$ and a spin structure $\sigma$ on $M$,
 there exists a
 unique spin invariant $I_{M,\sigma}(v)
\in \G^{T}$, dominating  spin refinements of quantum invariants
$\tau'_{M,\sigma}(\xi)$ at all roots of unity  $\xi$ whose
order is divisible by 4.

\end{Thm}

A similar statement concerning cohomological refinements is given
in Theorem \ref{main-coho}, where the unified invariant is an element
of $\G^{S_2}$ with $S_2=\{4(2n+1)\,|\, n\in \N\}$.

\begin{Cor}
Suppose $M\in M_n$.

(a) The
set of spin invariants $\tau_{M,\sigma}(\xi)$ is 
determined by their values at roots of unity
$\xi$ with ${\rm ord}(\xi)=\{ 2^k p^e\, | \,k\geq 2, e\in \N\}$
where $p\neq 1$ is an odd prime.

(b) The set of cohomological refinements
$\tau_{M,\sigma}(\xi)$ is determined
by their values at roots of unity
$\xi$ with ${\rm ord}(\xi)=\{ 2 p^e\, | \, e\in \N\}$
where $p\neq 1$ is an odd prime.
\end{Cor}

\v8

\subsection{The case $|H_1|=2$} Suppose  $M\in\cM_1$.
Let us consider the ring
$$\ZZ := \lim_{\overleftarrow{\hspace{2mm}n\hspace{2mm}}}
\frac{\Z[v^{\pm 1}]}{(-v^2;-v)_{2n}},$$
where
$$ (-v^2;-v)_{2n} := \prod_{i=2}^{2n+1} (1+(-v)^i)= (1-v^3)(1-v^5)\dots (1-v^{2n+1})
(1+q)(1+q^2) \dots (1+q^{n}).$$
Every $f(v) \in \ZZ$ can be written as, with $f_n(v) \in \Z[v^{\pm
1}]$,
$$f(v) = \sum_{n=0}^\infty f_n(v) \, (-v^2;-v)_{2n},$$
If
$\zeta$ is a root of unity of order either odd or
divisible by 4, then the evaluation
$\ev_\xi(f(v))=f(v)|_{v=\zeta}$ is well--defined.
For every
root $\xi$ of unity, one can choose a square root $\zeta$ of $\xi$
whose order is either odd or divisible by 4.

Let
$$\ZZZ := \lim_{\overleftarrow{\hspace{2mm}n\hspace{2mm}}}
\frac{\Z[v]}{(1+q)(1+q^2)\dots(1+q^n)}.$$ If $\zeta$ is a root of
unity of order divisible by $4$, then $f(\zeta)$ is well--defined
for $f(v)\in \ZZZ$.

\begin{Thm}\label{cor1}
For every  3--manifold $M\in \cM_1$, there exists a unique
invariant $I_M(v) \in \ZZ$,
such that for any root of unity $\xi$,
  $\ev_\xi(I_M(v))=\tau'_M(\xi)$. Moreover,
$I_M(v)=\sum_\sigma I_{M,\sigma}(v)$ where  $\sigma$ is a spin
structure on $M$, and
$I_{M,\sigma}(v)\in \frac{1}{1-v}\ZZZ$  dominates
refined quantum invariants  $\tau_{M,\sigma}(\xi)$.

\end{Thm}

 The ring $\ZZ$ is smaller than the ring $ {\G^S }$ of
Theorem \ref{Two}.
 By results of \cite{Ha1},   $\ZZ$
embeds in $\Z[[v-1]]$, via Taylor series (compare Proposition \ref{zz}
below). As a consequence, we will prove
\begin{Cor} \label{co2}
For $M\in \cM_1$ and the quantum invariants $\tau'_M$, the following
statements hold.

(a) The quantum invariants
at all roots  of
unity are algebraic integers.

(b) The quantum  invariants at any infinite set  of roots
of unity of odd
prime power  order determine the whole set of
quantum invariants.

(c) The
 Ohtsuki series of $M$, a formal power series in $q-1$, has
coefficients in $\Z[1/2]$.
If $\zeta$ is a root of
 unity of order  $p^d$ with $p$ an odd prime, then the Ohtsuki
 series at $\zeta$ converges $p$--adically to the quantum invariant at
 $\zeta$.

 (d)
 The Le--Murakami--Ohtsuki  invariant determines
the quantum invariants
 at all roots of unity.
\end{Cor}

The integrality of $\tau_{M,\sigma}$ for $\Z/p\Z$--homology
spheres at roots of order  $2p$, where $p$ is an odd {\em prime}
and
 $\sigma$ is  a cohomological class,  was studied
by Murakami in \cite{HM, HM1}.
\v8

\noindent
{\bf Example.} Suppose $M$ is obtained by surgery on the figure 8
knot with framing 2.
Then
$$ I_M(v) = \sum^\infty_{n=0} v^{-n(n+2)}(-v^2;-v)_{2n}$$

\noindent
{\it a) Spin refinement.}
Let $\sigma_0$ be the characteristic spin structure on $M$,
 $\sigma_1$ the other one.
$$ I_{M,\sigma_\e}(v)=\frac{1}{2(1-v)}
 \sum^\infty_{n=0} v^{-n(n+2)}\prod^n_{i=1}(1+q^i)
\left[\prod^n_{i=0} (1-v^{2i+1})- (-1)^{\e+n}
\prod^n_{i=0}(1+v^{2i+1})\right]
$$
Suppose that  ${\rm ord}(\zeta) \equiv 0 \pmod 8$, and
${\rm ord}(\zeta)/8\equiv\chi\pmod 2$, then $\ev_\xi(I_{M,\sigma_\e}(v))=
\tau_{M,\sigma_{\e+\chi}}(\zeta)$.
\v8

\noindent
{\it b) Cohomological refinement.}
Let  $\sigma_\e\in H^1(M,\Z/2\Z)$, and $\sigma_1$ be trivial.
$$ I_{M,\sigma_\e}(v)=\frac{1}{2(1-v)}
 \sum^\infty_{n=0} v^{-n(n+2)}\prod^n_{i=1}(1+q^i)
\left[\prod^n_{i=0} (1-v^{2i+1})+
(-1)^{\e+n} I\prod^n_{i=0}(1+v^{2i+1})\right] $$
Assume that ${\rm ord}(\zeta)=4k$  with odd $k$ and
$\zeta^{k^2}=(-1)^\chi I$, where $I$ is the unit complex number,
 then  $\ev_{\xi}(I_{M,\sigma_{\e}}(v))=
\tau_{M,\sigma_{\e+\chi}}(\zeta)$.

\subsection{Plan of the paper}
The paper is organized as follows. After recalling the definitions,
we collect the results on cyclotomic completions
of polynomial rings in Section 2.
Then we introduce the Laplace transform method for constructing
unified invariants.
Applying this method  to integral homology 3--spheres,
we get precise formulas for  Habiro's
unified invariants.
After that we apply this method to $M\in \cM_n$.
Here again the exact formula for the  Laplace transform
 implies  various above mentioned results.
In Section 6, we construct
   spin and cohomological
refinements of the unified invariant.

\subsection*{Acknowledgment} The first author wishes to express her
gratitude to Dennis Stanton for the significant simplification of the proofs
of Lemmas \ref{l1}, \ref{l2}.

\section{The colored Jones polynomial and the WRT invariant}
\label{defs} Let us first fix the notation. Throughout the paper,
$q=v^2$.
$$ \{n\} := v^n-v^{-n},
 \quad  \{n\}! \, :=
\prod_{i=1}^n \{i\} ,\quad  [n] :=\frac{\{n\}}{\{1\}}, \quad
\qbinom{n}{k} := \frac{\{n\}!}{\{k\}!\{n-k\}!}.$$ Let
$(a;t)_k:=(1-a)(1-at)\dots(1-at^{k-1})$ and for simplicity  $(q)_n:=(q;q)_n$.

\subsection{The colored Jones polynomial}

\newcommand{\RR} {\mathbf R}

 Suppose $L$ is a framed, oriented link
in $S^3$ with $m$ ordered components.
 For every positive integer $n$ there is a unique
irreducible $sl_2$--module $V_n$ of dimension $n$.
For positive integers $n_1,\dots,n_m$ one can define
the quantum invariant $J_L(n_1,\dots,n_m):=
J_L(V_{n_1},\dots,V_{n_m})$ known as the colored
Jones polynomial of $L$ (see e.g. \cite{Tu}).
Let us recall here a few well--known formulas.
For the unknot $U$ with 0 framing one has
\begin{equation} J_U(n) = [n]= \{n\}/\{1\}. \label{unknot}
\end{equation}
 If $L'$ is obtained from $L$
by increasing the framing of the $i$--th component by 1, then
\begin{equation}\label{framing}
J_{L'}(n_1,\dots,n_m) = q^{(n_i^2-1)/4} J_{L}(n_1,\dots,n_m).
\end{equation}
In general, $J_{L}(n_1,\dots,n_m) \in \BZ[q^{\pm 1/4}]$. However,
there is a number $a\in \{0,\frac{1}{4},\frac{1}{2},\frac{3}{4}\}$
such that $J_{L}(n_1,\dots,n_m) \in q^a\BZ[q^{\pm 1}]$.

\subsection{Evaluation map and Gauss sum}
Throughout this paper, let $\xi$ be  a primitive root of unity of
order $r$. We first define, for each $\xi$, the evaluation map
$\ev_\xi$, which replaces $q$ by $\xi$. Suppose that $r$  is odd,
then there exists an integer $2^*$, unique modulo $r$, such that
$(\xi^{2^*})^{2}=\xi$. For $f\in \BQ[v^{\pm 1}]$, we define
$$\ev_\xi f := f|_{v= \xi^{2^*}}.$$
If $r$ is even, then $\ev_\xi$ {\it depends
 on  a square root $\zeta$ of $\xi$}. We define
$$\ev_\xi f := f|_{v= \zeta}.$$

Suppose
  $f(v;n_1,\dots,n_m)$ is a function
of variables $v$ and integers $n_1,\dots,n_m$. Let
$$ {\sum_{n_i}}^\xi f := \sum_{n_i} \ev_\xi (f),$$
where in the sum  all the $n_i$ run the set of  numbers between $0$
and $2r-1$. Moreover, we define
$$ {\sum_{n_i}}^{\xi,\e} f := \sum_{n_i\equiv\e\pmod2} \ev_\xi (f),$$
where for $\e=0$ (resp. $\e=1$),
 the $n_i$ in the sum run the set of even (resp. odd)  numbers
between $0$ and $2r-1$.

 Variations of the Gauss sum are defined by
the following formulas. Fix a 4--th root of $\xi$. For $b\in \Z$ let
$$ \gamma_b(\xi):= {\sum_{n}}^\xi q^{b\frac{n^2-1}{4}}\, ,\;\;\;\;\;\;
 \gamma^\e_b(\xi):= {\sum_{n}}^{\xi,\e} q^{b\frac{n^2-1}{4}}\, .$$

Furthermore,
$$ F_L(\xi):= {\sum_{n_i}}^\xi
J_L(n_1,\dots,n_m)\prod_{i=1}^m [n_i]\, .$$
For any sequence $c=(c_1,...,c_m)\in (\Z/2\Z)^m$ we define
$$ F^c_L(\xi):= {\sum_{n_i}}^{\xi,c_i+1}
J_L(n_1,\dots,n_m)\prod_{i=1}^m [n_i]\, .$$
For $\e=0$ or $1$, let $F^\e_L(\xi):=F^{(\e,\e,...,\e)}_L(\xi)$.

\subsection{Quantum (WRT) invariants and their refinements}\label{def}
All 3--manifolds in this paper are supposed to be compact, closed and
oriented. Every link in a 3--manifold is framed, oriented, and has
components ordered.

 Suppose $M$ is an
oriented 3--manifold obtained from $S^3$ by surgery along a framed,
oriented link $L$. (Note that $M$ does not depend on the orientation
of $L$). Let $\sigma_+ $ (respectively, $\sigma_-$) be the number of
positive (resp. negative) eigenvalues of the linking matrix of $L$.
Suppose $\xi$ is a root of unity of  order $r$ together with a fixed
4--th root of it. Then the WRT (or quantum) $SU(2)$ invariant
\cite{Tu} is defined by
\begin{equation*}
\tau_M(\xi) =
\frac{F_L(\xi)}{(F_{U^{+}}(\xi))^{\sigma_+}\,
(F_{U^{-}}(\xi))^{\sigma_-} }.
\end{equation*}
For connected sum, one has $ \tau_{M\#N}(\xi) =\tau_{M}(\xi)
\tau_N(\xi).$

Suppose $\xi$ is a root of unity of odd order $r$. Then the quantum
$SO(3)$ invariant \cite{KM} is defined by
\begin{equation*}
\tau_M^{SO(3)}(\xi) :=
\frac{F^1_L(\xi)}{(F^1_{U^{+}}(\xi))^{\sigma_+}\,
(F^1_{U^{-}}(\xi))^{\sigma_-} }.
\end{equation*}

Let $L_{ij}$ be the $(i,j)$--entry of the linking matrix of $L$.
Any solution $c=(c_1,...,c_m)$
of the characteristic equation $L_{ij} c_j=L_{ii} \pmod 2$ defines
a spin structure $\sigma_c$ on $M$ \cite{KM}.
If the order of $\xi$ is divisible by 4, then there exists an invariant
of the pair $(M,\sigma_c)$ defined as follows.

\begin{equation}\label{dspin}
\tau_{M,\sigma_c}(\xi) =
\frac{F^c_L(\xi)}{(F_{U^{+}}(\xi))^{\sigma_+}\,
(F_{U^{-}}(\xi))^{\sigma_-} }.
\end{equation}

If the order of $\xi$ is $2\pmod 4$, then (\ref{dspin})
defines an invariant
of the pair $(M,\sigma_c)$, where $\sigma_c\in H^1(M,\Z)$
is determined by the solution $c$ of the following
equation  $L_{ij} c_j=0 \pmod 2$. Clearly,
$\tau_M(\xi)=\sum_\sigma \tau_{M,\sigma}(\xi)\,$ .

\subsection{Habiro's cyclotomic expansion of the colored Jones
polynomial}

For a link $L$ with $m$ components, define

$$ J'_L(n_1,\dots,n_m) := \frac{J_L(n_1,\dots,n_m)}{[n_1] \dots
[n_m]}.$$

Let $K$ be a knot with framing zero.
 Note that $J'_K(\l)\in \Z[q^{\pm 1}]$ for integer $\lambda \ge 1$.
In \cite{Ha}, Habiro proved that there exist $C_{K,k}(q)\in
\Z[q^{\pm 1}]$ such that \be\label{hab}
 J'_K(\l)=\sum^{\infty}_{k=0} C_{K,k}(q)\, (q^{1+\l})_k (q^{1-\l})_k\, .\ee
 The sum in (\ref{hab})
is finite, because the summands with $k\geq \l$ are zero. This
expansion is called the cyclotomic expansion of the colored Jones
polynomial. The non--trivial part here is that $C_{K,k}$'s are Laurent
polynomials in $q$ with integer coefficients.
\v8

\noindent {\bf Examples.}
Let $3_1$, $\bar 3_1$ and $4_1$ denote
the right--, left--handed trefoil and
the figure 8 knot.   We have
$$J'_{3_1}(\l)=\sum^{\infty}_{k=0} q^{-{k(k+2)}}(q^{1+\l})_k (q^{1-\l})_k\, ,\;\; \,
\;\;\;\;
J'_{\bar 3_1}(\l)=\sum^{\infty}_{k=0} q^{k}(q^{1+\l})_k (q^{1-\l})_k \, ,
$$
$$J'_{4_1}(\l)=\sum^{\infty}_{k=0} (-1)^k q^{-\frac{k(k+1)}{2}}(q^{1+\l})_k (q^{1-\l})_k \,
.$$
\noindent
{\bf Note.} The coefficients $C_{K,k}$ are computed for all twist knots
in \cite{Ma}.

More generally, we have the following.
\begin{pro}(Habiro) \label{link}
Let $L$ be an algebraically split link of $m$ components.
There exist $C_{L,\fk}(v)\in \Z[v^{\pm 1}]$ with $\fk=(k_1,\dots,k_m)$,
such that
$$J'_L(n_1,\dots,n_m)=\sum_{k\geq0}   \left(\sum_{\max k_i=k} C_{L,\fk}(v) \, (1-q)^l  \prod^{l}_{i=1}
\frac{(q^{1+n_i})_{k_i}(q^{1-n_i})_{k_i}}{(q^{k_i+1})_{k_i+1}}\right)
\frac{(q^{k+1})_{k+1}}{ (1-q)}\, .
$$
\end{pro}

\noindent {\bf Example.} Let $L$ be the $0$--framed Whitehead link.
The following formula was obtained by Habiro in \cite{Ha5}.
$$ J'_L(\l,\m)=\sum^{\infty}_{k=0} (-1)^{k} v^{-k(k+1)}(1-q)
\frac{(q^{1+\mu})_{k}(q^{1-\mu})_{k} }{(q^{k+1})_{k+1}}\;
(q^{1+\l})_{k}(q^{1-\l})_{k}$$
\v8

\section{Cyclotomic completions of polynomial rings}\label{cyc}
We present and modify some results of \cite{Ha1} here. Let $R$ be a
commutative ring with unit, and let $R[q]$ be the polynomial ring
over $R$. Recall that $\Phi_n(q)$  denotes the $n$--th cyclotomic
polynomial. If $S\subset \N$, we set $\Phi_S=\{\Phi_n(q)| n\in S\}$.
Let $\Phi^*_S$ denote the multiplicative set in $\Z[q]$ generated by
$\Phi_S$  and directed with respect to the divisibility relation.
The principal ideals $(f(q))\subset R[q]$ for $f(q)\in \Phi^*_S$
define a linear topology of the ring $R[q]$. In \cite{Ha1}, Habiro
defined the ($S$--) cyclotomic completion
 ring $R[q]^S$ as follows:
\be\label{rs} R[q]^S:=\lim_{\overleftarrow{\hspace{2mm}f(q)\in
\Phi^*_S\hspace{2mm}}} \;\;\;\frac{R[q]}{(f(q))}. \ee
For example, since the sequence $(q)_n$, $n\in \N$,
 is cofinal to $\Phi^*_\N$, we have
$$\Habiro\simeq\Z[q]^\N.$$
 Similarly, for $S=\{n\in \N\, | n\not\equiv 2 \pmod 4\}$
$$\ZZ\simeq\Z[v]^S.$$
 Note that if $S$ is finite, then
$R[q]^S$ is identified with the ($\prod \Phi_S$--)adic completion of $R[q]$.
In particular,
$$R[q]^{\{1\}}\simeq R[[q-1]], \quad
R[q]^{\{2\}}\simeq R[[q+1]].$$
Suppose $S' \subset S$, then $\Phi^*_{S'}\subset \Phi^*_S$, hence
there is natural map
$$ \rho^R_{S, S'}: R[v]^S \to R[v]^{S'}.$$
In particular, when $S=\{n\in \N\, | n\not\equiv 2 \pmod 4\}$ and
$S' = \{1\}$, the map $$\rho^\Z_{S,S'} : \ZZ \to \BZ[[v-1]]$$ is the
Taylor expansion.

Recall important results concerning  $R[v]^S$ from \cite{Ha1}. Two
positive integers $n, n'$ are called {\em adjacent} if and only if
$n'/n=p^e$ with
$e\in \Z$, for a prime $p$, such that
 the ring $R$ is $p$--adically separated.
A set of positive integers is
{\em connected} if for any two distinct elements $n,n'$ there is a
sequence $n=n_1, n_2 \dots, n_{k-1}, n_k= n'$ in the set, such that
any two consecutive numbers of this sequence are adjacent.
Theorem 4.2 of \cite{Ha1}
says that if $S$ is connected, then for any subset $S'\subset S$,
the natural map $ \rho^R_{S,S'}: R[v]^S \to R[v]^{S'}$ is an
embedding.

If $\zeta$ is a root of unity of order in $S$, then for every $f(v)\in R[v]^S$
the evaluation $\ev_\zeta(f(v))\in R[\xi]$ can be defined by sending 
$v\to\zeta$.
For a set $\Xi$ of roots of unity whose orders form a subset
$\cT\subset S$, one defines the evaluation

$$ \ev_\Xi: R[v]^S \to \prod_{\zeta \in \Xi} R[\zeta].$$

Theorem 6.1 of \cite{Ha1} shows that if
$R\subset \Q$, $S$ is connected, and there
exists $n\in S$ that is adjacent to infinitely many elements in
$\cT$, then $\ev_\Xi$ is injective.

\begin{pro} \label{zz}(a) The  Taylor expansion map
$T_{}: \ZZ \to \Z[[v-1]]$ is injective.

(b) For   a non--negative integer $k\neq 1$ and a prime number $p\neq 1$,
let
 $\cT_k=\{2^kp^e\,|\, e\in \N\}$.
 Suppose $f(v), g(v) \in \ZZ$ such that
$f(\zeta)= g(\zeta)$ for every $\zeta$ with
${\rm ord}(\zeta) \in \cT_k$, then $f(v)=g(v)$.

(c) The natural Taylor map $\ZZZ\to \BZ[I, 1/2][[1+q]] $, explained
in the proof, is injective. Here $I$ is the imaginary unit,
$I^2=-1$.
\end{pro}

\begin{proof} (a)
It is easy to see that $S=\{n\in \N\, | n\not\equiv 2 \pmod 4\}$  is
connected. Note that if $S'=\{1\}$  then $\Z[v]^{S'} = \Z[[v-1]]$.
Hence part (a) follows from the above mentioned Theorem 4.2 of
\cite{Ha1}.

(b) Since $k\neq1$, $2^k \in S$, and by assumption, $2^k$ is
adjacent to every element in $\cT$. Part (b) follows from the above
mentioned \cite[Theorem 6.1]{Ha1}.

(c) It is easy to see the set $S=\{n\in \N\,|\, n\equiv0\pmod 4\}$
is connected, and $\ZZZ \cong \Z[v]^{S}$. Hence $\rho^{\BZ}_{S,
\{4\}}$ is injective.
For the set $\{4\}$, we
have
$$ \Z[v]^{\{4\}} \simeq  \BZ[v][[1+v^2]] \simeq \BZ[v][[1+q]].$$

Using
$$ v = \sqrt q = \sqrt {-(1 - (1+q))}= I (1- (1+q))^{1/2} \in
\BZ[I][1/2][[1+q]]$$
we see that there is an embedding of $\Z[v]^{\{4\}}$ into $\BZ[I,
1/2][[1+q]]$, which, combined with $\rho^{\BZ}_{S, \{4\}}$, gives us
the injective Taylor map.\end{proof}

Let $S_k=\{2^k(2n+1)\,|\, n\in \N\}$.  Then  for every $f(v)
\in \Z[1/2,I][v]^{S_k}$, and any root of unity $\zeta$ of order in $S_k$,
the evaluation
map $\ev_\zeta:\Z[1/2,I][v]^S\to Z[1/2][\zeta]$ can be defined as follows:
$\ev_{\zeta}(v)=\zeta$,
$\ev_{\zeta}(I)=\zeta^{\rm ord(\zeta)/4}$.

Let us study for   $\G :=\Z[1/2][v]$ and $S=\{n\in \N\,|\,
n\not\equiv 2\pmod 4\}$, the completion
  $\G ^S$  mentioned in Introduction. Note that $S$
is not connected in Habiro sense for $R=\Z[1/2]$.
We have $S=\cup_{k\in \N, k\neq 1} S_k$.

\begin{pro}\label{zzz}

(a) One has  $$\G^S=\prod_{k\in\N,k\neq 1} \G^{S_k}\, .$$

(b) For   an integer $k\geq 2$ and an
odd prime number $p\neq 1$,
let
 $\cT_k=\{2^kp^e\,|\, e\in \N\}$.
 Suppose $f(v), g(v) \in \Z[1/2,I][v]^{S_k}$ such that
$\ev_{\zeta}(f(v))= \ev_\zeta(g(v))$ for every $\zeta$
with ${\rm ord}(\zeta) \in \cT_k$, then $f(v)=g(v)$.

(c) For   an odd prime number $p\neq 1$,
let
 $\cT=\{2^kp^e\,|\,k\in \N, e\in \N\}$.
 Suppose $f(v), g(v) \in \G^{S}$ such that
$\ev_{\zeta}(f(v))= \ev_\zeta(g(v))$ for every $\zeta$
with ${\rm ord}(\zeta) \in \cT$, then $f(v)=g(v)$.
\end{pro}

\begin{proof}
(a) Let us first prove
 that if $n=2^k n'\in S_k$ and $m=2^l m'\in S_l$ with 
$k\neq l$, then $(\Phi_n,\Phi_m)=(1)$ in $\G=\Z[1/2][v]$.
Indeed, if $n'\neq m'$, then $n$ and $m$ are not adjacent, hence
$(\Phi_n,\Phi_m)=(1)$ in $\Z[v]$ and 
the claim holds. If $n'=m'$, then $n/m=2^{k-l}$, hence in $\Z[v]$
one has $2\in (\Phi_n,\Phi_m)$, which implies
the claim,  since $2$ is invertible in $\G$.

Suppose $f\in \Phi^*_S$, then $f=\prod f_k$ with $f_k\in \Phi^*_{S_k}$.
Hence the $f_k$'s are pairwise coprime. By the Chinese remainder theorem,
$$\frac{\G}{(f)}=\prod_k \frac{\G}{(f_k)}\, .$$
Taking the inverse limit, we get (a).

(b) It is easy to see that $S_k$ is connected
in Habiro's sense for the ring $\Z[1/2,I]$. Hence
 Part (c) follows by adopting the proof of \cite[Theorem
6.1]{Ha1} to the ring $Z[1/2,I]$, which is straightforward.

(c) is an easy consequence of (a) and (b).

 \end{proof}


\section{Laplace transform}

In this section we introduce the Laplace transform method.
For simplicity, we  restrict to homology spheres  obtained
by  surgery on knots, the general
case can be obtained by applying
the same arguments to each component of the link.

\subsection{Quantum invariants for knot surgeries}
For any non--zero integer $b$, and a knot $K$,
let $M=S^3(K_b)$ be a homology sphere obtained by surgery on $K$ with
 framing $b$. Assume that $\xi$ is
 a primitive $r$--th root of unity and $r$ is even.
The  quantum $SU(2)$ invariant of $M$ is defined as follows:
 \be\label{inv} \tau_M(\xi)= \frac{{\sum\limits_\l}^\xi
\; q^{\frac{b(\l^2-1)}{4}} \;\,
(1-q^\l)(1-q^{-\l})J'_K(\l)}{{\sum\limits_{\l}}^\xi
q^{\frac{sn(b)(\l^2-1)}{4}}\;\, (1-q^\l)(1-q^{-\l})}\; , \ee where
$sn(b)$ is the sign of $b$. To be precise, one needs to fix a 4--th
root of $\xi$.
Note that when computing the Jones polynomial
of a knot (or a link) further in this paper, we always
assume that its framing
is zero. However, in the formula for the quantum invariant,
 framing is taken into account by means of the factor $q^{b(\l^2-1)/4}$.

Substituting Habiro's formula (\ref{hab}) into (\ref{inv}) we get
\be\label{inv1} \tau_M(\xi)= \frac{{\sum\limits_{\l}}^\xi\;
q^{\frac{b(\l^2-1)}{4}} \;\,\sum\limits^{\infty}_{n=0}\,
  C_{K,n} F_n(q^\l,q)}{{\sum\limits_{\l} }^\xi\;
q^{\frac{sn(b)(\l^2-1)}{4}}\;\, F_0(q^\l,q)}\, , \ee where
$F_n(q^\l,q)=(q^\l)_{n+1}(q^{-\l})_{n+1}$.

Suppose  $r$ is odd. Then, taking the sums over odd $\l$ in the
numerator and the denominator of (\ref{inv}) we get the  $\tau^{SO(3)}_M$.
 In this case, there is no need to fix a  4--th root of
$\xi$.

\subsection{Laplace transform method}

The main idea behind the Laplace transform method is to interchange
the sums over $\l$ and $n$ in (\ref{inv1}) and regard
$\sum^{\xi}_{\l } q^{b(\l^2-1)/4}$ as an operator (called Laplace
transform) acting on  $F_n(q^\l,q)$. (Recall that $\int e^{-ax^k} f(x) dx$
is called Laplace transform of $f$ of order $k$. Our  sum is a discrete
version of the Laplace transform of the second order.)

More precisely, after interchanging the sums in the numerator of
(\ref{inv1}) we get
$$\sum^{r-1}_{n=0}C_{K,n}(q){\sum_{\l }}^\xi
q^{b\frac{(\l^2-1)}{4}} F_n(q^\l,q)\, .$$ Now  observe, that
$F_n(q^\l,q)=(q^\l)_{n+1}(q^{-\l})_{n+1}$ is a Laurent polynomial in
two variables $q^\l$ and $q$. The Laplace transform does not affect
$q$, and  we only need  to compute the action of the  Laplace
operator on $q^{a\l}$.

Suppose the greatest common divisor of $b$ and $r$ is 1 or 2,
and $r$ is even. A
simple square completion argument shows that
$${\sum_{\l}}^\xi q^{\frac{b(\l^2-1)}{4}}\, q^{a\l}=
\xi^{-\frac{a^2\, b^*}{(b,r)}}\,
\gamma_{b}(\xi) \, $$ where $b^*$ is an integer such
that $b^* b =(b,r) \pmod r$.
Summarizing the previous discussion, we get
$${\sum_{\l}}^\xi
q^{\frac{b(\l^2-1)}{4}}\, F_n(q^\l,q)= \ev_\xi(L_{b;\l}(F_n(q^\l,q)))\;
\gamma_{b}(\xi)\, .$$
Here $L_{b;\l}(F)$ is the Laplace transform of $F$,
which is defined as follows. Suppose  $F$ is a formal power series
in $q^{\pm 1}$ and $q^{\pm \l}$. Then $L_{b;\l}(F)$ is obtained from $F$
by  replacing every $q^{a\l}$  by $q^{-a^2/b}$. The evaluation map
$\ev_\xi$ converts $q^{1/b}$ to $(\xi^{1/(b,r)})^{b*}$. Note that
while $\ev_\xi$ might depend on $r$, the Laplace transform $L_{b;\l}$ does
not. Also if $b=1$ or $b=2$, then $\ev_\xi$ does not depend on
$r$. In these cases, $\ev_r(q^{1/b})=\xi^{1/b}$.

If $r$ is odd and $(b,r)=1,2$, we can define the Laplace transform
$L_{b;\l}:\BZ[q^{\pm \l},q^{\pm 1}]\to \BZ[q^{\pm 1/b}]$
as a $\BZ[q^{\pm 1}]$--linear operator
 sending  $q^{a\l} \mapsto q^{-a^2/b}$.
 In this case, we have
$${\sum_{\l }}^{\xi,1}
q^{\frac{b(\l^2-1)}{4}}\, F_n(q^\l,q)= \ev_\xi(L_{b;\l}(F_n(q^\l,q)))\;
\gamma^1_{b}(\xi)\; .$$


As a result,  we have  closed formulas for  quantum
invariants in terms of the Laplace transform.

\begin{thm}\label{le}
Let $M=S^3(K_b)$ and $(b,r)=1$ or 2.  Then
$$\tau_M(\xi)=\frac{1}{2(1-\xi^{-sn(b)})}
\frac{\gamma_{b}(\xi)}{\gamma_{sn(b)}(\xi)}\;
\sum^\infty_{n=0} C_{K,n} \ev_\xi(L_b(F_n))\, ,$$
$$\tau^{SO(3)}_M(\xi)
=\frac{1}{2(1-\xi^{-sn(b)})}
\frac{\gamma^1_{b}(\xi)}{\gamma^1_{sn(b)}(\xi)}\;
\sum^\infty_{n=0} C_{K,n} \ev_\xi(L_b(F_n))\, .$$

\end{thm}

\section{Habiro theory}
In this section we show how Theorem \ref{le}
can be used to compute Habiro's
 unified  invariant of integral homology spheres.

\subsection{Knot surgeries}
Any knot surgery with framing $b=\pm 1$ yields an integral homology sphere.
Combining
Theorem \ref{le} with Lemma \ref{l1} below we get the following
theorem.

\begin{thm}\label{le1} (Habiro)
For $M_\pm=S^3(K_{\pm 1})$, there exists a unique invariant
$$ I_{M_+}(q)=
\sum^\infty_{n=0} (-1)^{n}  q^{-\frac{n(n+3)}{2}} C_{K,n}(q)
\frac{(q^{n+1})_{n+1}}{1-q}\,\in \Habiro ,$$
$$ I_{M_-}(q)=
\sum^\infty_{n=0} C_{K,n}(q) \frac{ (q^{n+1})_{n+1}}{1-q}\, \in \Habiro\, ,$$
such that $I_{M_{\pm}}(\xi)=\tau_{M_\pm}(\xi)=\tau^{SO(3)}_{M_\pm}(\xi)$.
\end{thm}

\noindent

\v8

\noindent {\bf Examples.} Denote by $3_1$ and $4_1$ the Poincare
sphere and
 the $3$--manifold obtained by framing 1
surgery on figure 8 knot. By Theorem \ref{le1}, we have
$$I_{3_1}(q)= \frac{q}{1-q}\sum^\infty_{k=0}(-1)^k
q^{-\frac{(k+2)(3k+1)}{2}} (q^{k+1})_{k+1}$$
$$I_{4_1}(q)= \frac{q}{1-q}\sum^\infty_{k=0}(-1)^k
q^{-(k+1)^2} (q^{k+1})_{k+1}$$

\begin{lem} \label{l1}
\be\label{lap1}L_{-1}((q^\l)_{k+1}(q^{-\l})_{k+1})=2(q^{k+1})_{k+1}\, . \ee
\be\label{lap2}
L_1((q^\l)_{k+1}(q^{-\l})_{k+1})=2(-1)^{k+1} q^{-\frac{(k+2)(k+1)}{2}}
(q^{k+1})_{k+1}\, .\ee
\end{lem}

\begin{proof}
Let $L_b: \Z[x^{\pm 1},q^{\pm 1}]\to \Z[q^{\pm 1/b}]$
be a $\Z[q^{\pm 1}]$--linear operator sending $x^a\mapsto q^{-a^2/b}$.
Then for  $F(x,q)\in \Z[x^{\pm 1},q^{\pm 1}]$ we have
$$L_{-b}(F(x,q))= q^{k(k+1)}\, L_b(F(x,q^{-1}))\, .$$
Using this formula we can deduce (\ref{lap2})  from (\ref{lap1}).

Let us prove (\ref{lap1}). For this, we split
$$F_k(q^\l,q)=S_k(q^\l,q)+T_k(q^\l,q)$$
with $S_k(q^\l,q)=(q^\l)_{k+1}(q^{-\l+1})_k$ and
$T_k(q^\l,q)=-q^{-\l} (q^\l)_{k+1}(q^{-\l+1})_k$.
Then $S_k(q^{-\l},q)=T_k(q^\l,q)$ implies
$L_{b;\l}(S_k)=L_{b;\l}(T_k)$
for any $b$. Therefore, we have to look at one of them only.

Further, by the $q$--binomial theorem (eq. (II.4) in \cite{GR}) we get
$$S_k(q^\l,q)=(-1)^k q^{-k\l}q^{k(k+1)/2}(q^{\l-k})_{2k+1}=$$
$$
(-1)^kq^{\frac{k(k+1)}{2}}\sum^{2k+1}_{j=0} (-1)^j\left[\begin{array}{c}
2k+1\\ j\end{array}\right]_q q^{\frac{j(j-1)}{2}}q^{-kj}q^{(j-k)\l}\, $$
where
$$\left[\begin{array}{c}n\\k\end{array}\right]_q
=\frac{(q)_n}{(q)_k(q)_{n-k}}\, .
$$
Taking the Laplace transform we have
$$L_{-1}(S_k(q^\l,q))=(-1)^k q^{\frac{3k^2+k}{2}}\sum^{2k+1}_{j=o}\frac{
(q^{-2k-1})_j}{(q)_j}q^{j^2+j-jk}\, .$$
The result follows now by applying the
 Sears--Carlitz transformation (eq. (III.14)
in \cite{GR}) for terminating $_3\phi_2$ series with specializations
$a=q^{-2k-1}$, $b,c\to \infty$, $z\to q^{k+2}$.
\end{proof}

\subsection{Link surgeries}\label{2.2}
Suppose  $M$ be an integral homology
sphere.  Without loss of generality, we  can assume that
$M$ is obtained
by surgery on  an algebraically split link $L$ in $S^3$
 with framings $\pm 1$.
Suppose that the  first $\sigma_+$ components have framing $+1$,
and the others $-1$. Substituting cyclotomic expansion
of the colored Jones polynomial (given in Proposition \ref{link})
into (\ref{def})
and applying the Laplace transform method to each component of $L$,
we derive the following formula for the
  unified  invariant of $M$.
\begin{thm} (Habiro) For $M$ as above, there exists a unique invariant
\be\label{inv11} I_M(q)= \sum_{k=0}^\infty   \left(\sum_{\max
k_i=k} C_{L,\fk}(v)
 \prod\limits^{\sigma_+}_{i=1}
(-1)^{k_i} q^{-\frac{k_i(k_i+3)}{2}}\right) \frac{(q^{k+1})_{k+1}}{
(1-q)} \, \in \Habiro \ee
such that $I_{M}(\xi)=\tau_{M}(\xi)=\tau^{SO(3)}_{M}(\xi)$.
\end{thm}


\v8

\section{Rational homology $3$--spheres with $H_1(M)=(\Z/2\Z)^n$}\label{so3}

In this section we define the unified invariant for $M \in\cM_n$.

\subsection{Normalization}
Suppose that the order of $\zeta$ is divisible by 4.
It's easy to show that the quantum invariant of $\RP^3$, which is
obtained by sugery on the unknot with framing 2, is given by
$$ \tau_{\RP^3}(\xi) = \frac{\gamma_{2}(\xi)}{(1+\zeta^{-1})\,
\gamma_{1}(\xi)} = \frac{\gamma_{-2}(\xi)}{(1+\zeta)\,
\gamma_{-1}(\xi)}=\frac{\zeta^{-1/2}\sqrt2}{(1+\zeta^{-1})}.$$ For
$M\in \cM_n$, we will use a normalization such that the connected
sum of $n$ projective spaces $\RP^3$ takes value 1:
$$ \tau'_M (\xi) := \frac{\tau_M(\xi)}{(\tau_{\RP^3}(\xi))^n}.$$

For    an odd root of unity $\xi$, we put
$$ \tau'_M (\xi):= \frac{\tau^{SO(3)}_M(\xi)}{(\tau^{SO(3)}_{\RP^3}(\xi))^n},
\;\;\;\;{\rm with}\;
\;\;\;\tau^{SO(3)}_{\RP^3}(\xi)=\frac{\gamma^1_{2}(\xi)}{
(1+\zeta^{-1})\gamma^1_{1}(\xi)}\, .$$

\subsection{Diagonalization}

 Recall that {\em linking pairing} on a finite Abelian group $G$ is a
non--singular symmetric bilinear map from $G\times G$ to $\BQ/\BZ$.
Two linking pairing $\nu, \nu'$ on respectively $G,G'$ are
isomorphic if there is an isomorphism between $G$ and $G'$ carrying
$\nu$ to $\nu'$. With the obvious block sum, the set of equivalence
classes of linking pairings is a semigroup.

One type of linking pairing is given by  non--singular square
symmetric matrices with integer entries: any such $n\times n$ matrix
$A$  gives rise to a linking pairing $\phi(A)$ on $G= \BZ^n /A
\BZ^n$ defined by $\phi(A)(v,v') = v^t A^{-1} v' \in \BQ \mod \BZ$,
where $v,v'\in \BZ^n$. If there is a {\em diagonal} matrix $A$ such
that a linking pairing $\nu$ is isomorphic to $\phi(A)$, then we say
that $\nu$ is {\em of diagonal type}. It is known that if the
linking pairing of a 3--manifold is of diagonal type, with diagonal
entries $d_1, d_2,\dots, d_k$, then $M$ can be obtained
 by surgery along an algebraically split link, with framings
$d_1, d_2, \dots, d_k$ on $k$ components and framings $\pm 1$ on the
others, (see \cite{Oh, Le;last}).

\begin{lem}\label{top} Suppose $H_1(M, \BZ) = (\BZ/2)^n$.
Then $M\#\BR P^3$ can be obtained from $S^3$ by surgery on an
algebraically split link with framing $2$ on $n+1$ components and
framings $\pm 1$ on the others.
\end{lem}

\begin{proof}
The generators of the semi--group of linking pairings are known, see
\cite{Kojima,Wall}.
 Since subgroups of $H_1$ are of the form
$(\BZ/2\Z)^n$ only, from the list of generators we see that linking
pairing $\phi$ on $H_1(M,\BZ)$ must be the block sum of linking
pairings, each of the form $\phi(2)$ or $E_0^1$ in the notation of
\cite{Kojima}. Here $\phi(2)$ is considered as the $1\times 1$
matrix with the only entry $2$, and $E_0^1$ is a linking pairing on
$\BZ/2\times \BZ/2$.

Note that $\phi(2)$ is the linking pairing of $\BR P^3$. One of the
relations among generators says that $E_0^1 \oplus \phi(2) = \phi(2)
\oplus \phi(2) \oplus \phi(2)$. Thus, after adding a copy of $\BR
P^3$ the linking pairing of $M$ is of diagonal type, with entries
$2$ on the diagonal, and we get the result.
\end{proof}

\subsection{Unified invariant}\label{gener}

For any $M\in \cM_n$  by Lemma \ref{top}, $M'=M\#\BR P^3$ can be
obtained by surgery on an algebraically split $m$ component link
$L$. Let us assume that the first $s_+$ components of $L$ have
framing $2$, the next $s_-$ components have framing $(-2)$, the next
$l_+$ components are 1--framed, and the last components are
$(-1)$--framed. Although we can avoid the $-2$ framing, we add this
framing for the convenience of calculation. Note that $s_+ +s_-=n+1$,
since $M'\in \cM_{n+1}$.
The unified invariant of
$M$ is defined by
$$
I_{M}(v)= \frac{(1+v)^{n+1}}{1-q}\;
\sum_{k\geq 0} \; (q^{k+1})_{k+1} \; \times$$
$$ \left(\sum_{\max
k_i=k} C_{L,\fk}(v)\,\prod\limits^{s_+}_{i=1}
(-v)^{-k_i} B_{k_i} (v)\prod\limits^{n+1}_{i=s_+ +1} B_{k_i}(v)
 \prod\limits^{n+1+l_+}_{i=n+2}
(-1)^{k_i} q^{-\frac{k_i(k_i+3)}{2}}\right)
 \,
$$
where $B_{k}(v)=\left(\prod^k_{i=0} (1+v^{2i+1})\right)^{-1}$.

\subsection{Proof of Theorem \ref{Two}}
Let us first show that $I_M(v)\in {\G^S }$. Indeed,
 the denominator of $I_M$
contains only $\Phi_{4i+2}(v)$ for $i\in\BN$, which are invertible modulo
 $\Phi_s(v)$ for any $s\in S$ in $\Z[1/2][v]$ (compare the proof of 
Proposition \ref{zzz}, Part (a)).

We next  show that $\ev_\xi(I_M(v))=\tau'_{M'}(\xi)$. Observe
that $\tau'_{M'}(\xi)=\tau'_M(\xi)$ by the definition.
The proof is  an application of the Laplace transform method.
 Using Lemma \ref{l2} below we have
$$\ev_\xi\left(\frac{L_{2;\l}[(q^\l)_{k+1}(q^{-\l})_{k+1}]\, }{2(1-q^{-1})}
\right)\;\frac{\gamma_{2}(\xi)}{ \gamma_{1}(\xi)}
=\ev_\xi\left((-v)^{-k}\, (-v^2;-v)_{2k}\right)\, \tau_{\RP^3}(\xi)$$
$$
\ev_\xi\left( \frac{L_{-2;\l}[(q^\l)_{k+1}(q^{-\l})_{k+1}]\,
}{2(1-q)\,}\right) \;\frac{\gamma_{-2}(\xi)}{ \gamma_{-1}(\xi) }=
\ev_\xi\left((-v^2;-v)_{2k}\right)\, \tau_{\RP^3}(\xi)$$ In addition,
we use the following identity
$$\frac{(1-v)(-v^2;-v)_{2k}}{(q^{k+1})_{k+1}}= B_k(v)$$
 whose
proof is left to the reader.
Uniqueness of $I_M$  follows from Proposition \ref{zzz}, part $(c)$.
\qed

\begin{lem}\label{l2}
$$L_{-2;\l}(F_k(q^\l,q))=2(1-v)(-v^2;-v)_{2k}$$
$$L_{2;\l} (F_k(q^\l,q))=2(-1)^{k+1}v^{-k-1}(1-v)(-v^2;-v)_{2k}$$
\end{lem}
\begin{proof}
Recall that $L_{\pm 2;\l}(q^{a\l})=v^{\mp a^2}$.
We proceed by proving the first formula.
By the $q$--binomial theorem we get
$$L_{-2;\l}(F_k(q^\l,q))=
2 (-1)^kq^{k^2+k/2}\sum^{2k+1}_{j=0} \frac{(q^{-2k-1})_j}{(q)_j}
q^{j+j^2/2}\, .$$
The Sears--Carlitz transformation (eq. (III.14) in \cite{GR}) with
$a=q^{-2k-1}$, $c=-q^{-k}$, $z=q^{k+3/2}$ and $b\to \infty$
reduce this sum to $_2\phi_1(-q^{-k-1/2}, q^{-k};q^{-k+1/2},q)$
which can be computed by the $q$--Vandermode formula (eq. (II.6) in \cite{GR}).
The result follows. Note that $L_{2;\l}$ can be computed from $L_{-2;\l}$
by the same argument as in the proof of Lemma \ref{l1}.
\end{proof}

\noindent
{\bf Proof of Corollary \ref{cor-gen}.} Part $(a)$ follows, since
$\ev_\zeta(\G^S)=\Z[1/2][\zeta]$ if  ${\rm ord}(\zeta)\in S$.
Part $(b)$ is  the direct consequence of Proposition \ref{zzz}
$(c)$.

\subsection{Case $|H_1|=2$}

It is well--known that any $M\in \cM_1$ can be obtained by surgery
on an algebraically split link $L$.
 Let us assume that  $L$ is a link of $(m+1)$ components  numbered
by $0,1,\dots,m$, where the $0$--th component has framing $\pm
2$, the next $s$ components have framing 1, and the remaining ones
have framing $-1$.
The following
Proposition
 is the direct consequence of Theorem \ref{Two} and implies the part
of Theorem \ref{cor1} concerning the non--refined invariants.

\begin{pro}\label{h2} For $M_\pm$ as above, there exists an invariant
$$I_{M_+}(v)=\sum_{k_i\geq 0}   \left(\sum_{\max
k_i=k} C_{L,\fk}(v)\, (-v)^{-k_0}\prod_{i=k_0+1}^{k} (1+v^{2i+1})
 \prod\limits^{s}_{i=1}
(-1)^{k_i} q^{-\frac{k_i(k_i+3)}{2}}\right) (-v^2;-v)_{2k}
 \,
$$
and
$$I_{M_-}(v)=\sum_{k_i\geq 0}   \left(\sum_{\max
k_i=k} C_{L,\fk}(v)\prod_{i=k_0+1}^{k} (1+v^{2i+1})
 \prod\limits^{s}_{i=1}
(-1)^{k_i} q^{-\frac{k_i(k_i+3)}{2}}\right) (-v^2;-v)_{2k}
 \, ,
$$
such that $\ev_{\xi}(I_{M_\pm}(v))=\tau'_{M_\pm}(\xi)$
for any root of unity $\xi$. Moreover,
$I_{M_\pm}\in \ZZ$.
\end{pro}

\v8
\noindent
{\bf Example.} Let $L$ be the Whitehead link with framings $2$ and $-1$ and
 $M=S^3(L)$.
$$I_M(v)=\sum^\infty_k v^{-k(k+2)} (-v^2;-v)_{2k} $$
\v8


\noindent
\subsection{Proof of Corollary \ref{co2}}
(a)  follows from the fact that $\tau'_M(\xi) =\ev_\xi(I_M)$, and
for any $f\in \ZZ$, $ \ev_\xi(f) \in \BZ[\zeta]$.

(b) An infinite set of roots of unity of orders odd prime powers is
a set of the form $\cT_k$ in Proposition \ref{zz}, Part (b). Hence
Proposition \ref{zz} (b) implies the result.

(c) The Ohtsuki series is just the Taylor expansion of $I_M$.
 Observe that the Taylor series
in $(1-v)$ can be converted into a formal power series
in $(1-q)$ by
$$ v-1=(1+(q-1))^{1/2}-1=\sum^\infty_{n=1}{1/2 \choose n}(q-1)^n\, ,
$$
where ${1/2\choose n} \in \Z[1/2]$.

(d) The LMO invariant determines the Ohtsuki series via $sl_2$
weight system, see \cite{Oh}. \qed

\section{Refinements}
In this section we show that the Laplace transform method can
effectively be used  to define refinements of
the unified invariant.

\subsection{Spin case }\label{refdef}
Suppose  $M\in \cM_n$ can be obtained by surgery
on an algebraically split $m$ component
link $L$ as described in Section \ref{gener}, where
the first $n$ components of $L$ are $(\pm 2)$--framed, and the
remaining components are $(\pm 1)$--framed.
Suppose $\sigma_c$ is a spin structure on $M$
corresponding to the solution
$c$ of the characteristic equation. 
There are $2^n$ spin structures on $M$.

Let $$B_k(x,v)=
\frac{1}{2}\left(\prod\limits^k_{i=0}(1-v^{2i+1}) + x
\prod\limits^k_{i=0}(1+v^{2i+1})
\right)\prod\limits^k_{i=0}\frac{1}{1-q^{2i+1}}\, .$$

Then  we define the unified invariant of $(M,\sigma_c)$ as follows.
\be\label{invar}
I_{M, \sigma_c}(v)= \frac{(1+v)^{n}}{1-q}\;
\sum_{k\geq 0} \; (q^{k+1})_{k+1} \;
 \left(\sum_{\max
k_i=k} C_{L,\fk}(v)\,\prod\limits^{s_+}_{i=1}
(-v)^{-k_i}\times\right. \ee
$$\left. B_{k_i}((-1)^{c_i+k_i}, v)
\prod\limits^{n}_{i=s_+ +1} B_{k_i}((-1)^{c_i+1},v)
 \prod\limits^{n+l_+}_{i=n+1}
(-1)^{k_i} q^{-\frac{k_i(k_i+3)}{2}}\right)
 \,
$$

More generally,
 $M'=M\#\BR P^3$ can  be obtained by surgery
on an algebraically split link, i.e. $I_{M',\sigma'}$ is  defined
by (\ref{invar}). Assume that $\sigma'$ and $\sigma''$ are the two spin
structures on $M'$ whose restrictions to $M$ coincide with $\sigma$.
Then we define the unified invariant of $(M,\sigma)$ as follows.
\be\label{nondiag}
I_{M,\sigma}:=I_{M',\,\sigma'}+I_{M',\,\sigma''}\, . \ee
If the surgery matrix of $M$ is diagonalizable, then (\ref{nondiag})
coincides with (\ref{invar}), since $I_{M,\sigma}$ is multiplicative
with respect to the connected sum. Indeed, we have
$$ I_{M,\sigma}= I_{M,\sigma}\left(
I_{\BR P^3,\, \sigma_0} + I_{\BR P^3,\, \sigma_1}\right)=I_{M,\sigma}\left(
\frac{1}{1-v^{-1}} + \frac{-v^{-1}}{1-v^{-1}}\right)$$

Suppose  $ h\in H^1(M,\Z/2\Z)$
 assigns $1$ to all $n$ generators of $H_1(M,\Z/2\Z)$.
Recall that $S_k=\{2^k(2n+1)\,|\,n\in\N\}$.
Suppose $T=\cup_{k>2, k\in \N} S_k$ and 
 $T_1=\cup_{k>3, k\in \N} S_k$.

\begin{pro}\label{ququ}
For $M\in \cM_n$ and a spin structure $\sigma$ on $M$,
there exists an unique invariant $I_{M,\sigma}(v)
\in \G^{T}$, 
such that
for any root of unity $\zeta$ with
  ${\rm ord(\zeta)}\in T_1 $,
  $\ev_\xi(I_{M,\sigma}(v)) = \tau'_{M,\sigma}(\xi)$.
If ${\rm ord(\zeta)}\in S_3 $,
then  $\ev_\xi(I_{M,\sigma}(v)) = \tau'_{M,\sigma+h}(\xi)$.
\end{pro}

  Proposition \ref{ququ} implies  Theorem \ref{main_spin}.

\begin{proof}
Let us first show that $I_{M,\sigma_\e}(v)
\in \Z[1/2][v]^{T}$.
Recall that $(\Phi_n)+(\Phi_m)=(1)$ in $\Z[1/2][v]$ if and only if
$n$ and $m$ are not adjacent in Habiro sense.
The denominator of (\ref{invar})
consists of all $\Phi_i(v)$ with $i|4j+2$, $j\in \BN$.
But such $i$ is not adjacent to elements of $T$, i.e.
$\Phi_i(v)$ are invertible in $\G^{T}$ for all $i$.

We next show that the evaluation $I_{M,\sigma_\e}(v)$
coincides with the renormalized refined Witten--Reshetikhin--Turaev
 invariant.
We again use the Laplace transform method.

For $k=3$, we define
the refined Laplace transforms
$L^\e_{\pm 2;\l}:\BZ[q^{\pm \l},q^{\pm 1}]\to \BZ[q^{\pm 1/2}]$
as $\BZ[q^{\pm 1}]$--linear operators with
$$L^0_{\pm 2;\l}(q^{a\l})=\left\{
\begin{array}{ll}
 v^{\mp a^2}  &
{\rm for}\;\; a\;\;\; {\rm odd} \\
 0& {\rm otherwise} \end{array} \right.$$
$$L^1_{\pm 2;\l}(q^{a\l})=\left\{
\begin{array}{ll}
 v^{\mp a^2}  &
{\rm for}\;\; a \;\;\;{\rm even} \\
 0& {\rm otherwise} \end{array} \right.$$
For $k>3$,   the
previous definitions of $L^0_{\pm 2;\l}$ and $L^1_{\pm 2;\l}$
should be interchanged.
Then, for any  Laurent
polynomial $F(q^\l,q)$ and any root of unity $\zeta$ with
${\rm ord(\zeta)}\in S_k$,
 the following equation holds.
\begin{equation}
\gamma_{\pm 2}(\xi)\;\;\ev_\xi\left( L^\e_b (F(q^\l,q))\right) =  {\sum_{
\l}}^{\xi, \e} q^{\pm (\l^2-1)/2}\, F(q^\l,q)\,
\label{new1}\end{equation}
Moreover, $\gamma_{\pm 2}(\xi)=\gamma^1_{\pm 2}(\xi)$ if $k=3$
and $\gamma_{\pm 2}(\xi)=\gamma^0_{\pm 2}(\xi)$ if  $k>3$.
In addition, we have
$$(L^1_{\pm2}-L^0_{\pm2})(q^{a\l})=(-1)^a (-1)^{\chi+1}
L_{\pm2}(q^{a\l})=(-1)^{\chi+1}(-1)^a v^{\mp a^2}=
(-1)^{\chi+1} L_{\pm2}|_{v\to -v} (q^{a\l})$$
where $\chi=1$ if $k=3$ and zero otherwise.
This allows us to express  the refined Laplace transforms
in terms of $L_{\pm 2;\l}$:
$$L^\e_{\pm 2}=\frac{1}{2}\left(\; L_{\pm 2}+
  (-1)^{\e+\chi}L_{\pm 2}|_{v\to -v}
\;\right)\, . $$

Let us recall
that
 $ F^0_{L}(\xi) = F_{L}(\xi)$
 (compare \cite{KM}, \cite{B}).

The assertion in the diagonal case follows now from the next two formulas:
$$
\ev_\xi\left( \frac{L_{2;\l}(F_k)|_{v\to -v}\, }{2(1-q^{-1})}\right)
\;\frac{\gamma_{2}(\xi)} {\gamma_{1}(\xi)}=\;
-\frac{\zeta^{-k}(1+\zeta)}{1-\zeta}\; (-\zeta^2;\zeta)_{2k}\,
\tau_{\RP^3}(\xi)$$
$$ \ev_\xi\left(
\frac{L_{-2;\l}(F_k)|_{v\to -v}\, }{2(1-q)}\right)\,
\frac{\gamma_{-2}(\xi)}{ \gamma_{-1}(\xi) }=\;
\frac{1+\zeta}{1-\zeta}\;(-\zeta^2;\zeta)_{2k}\,
\tau_{\RP^3}(\xi)$$

Observe that $\tau'_{M,\sigma}(\xi)=
\tau'_{M',\sigma'}(\xi)+ \tau'_{M',\sigma''}(\xi)$.
 Uniqueness of $I_{M,\sigma}$ is provided by Proposition
\ref{zzz}, Parts $(a),(b)$. 
\end{proof}

Note that if $M\in\cM_1$, then $I_{M,\sigma}\in \frac{1}{1-v}\ZZZ$.
Thus we have the part of Theorem \ref{cor1} concerning spin refinements.

\subsection{Cohomological case}
Let us first assume that  $M\in\cM_n$ 
can be obtained by surgery on an algebraically split link $L$
as above.
Suppose $\sigma_c\in H^1(M,\Z/2\Z)$ is induced by the solution
$c$ of the following equation: $L_{ij} c_j=0\pmod 2$.

We define
\be\label{invar1}
I_{M, \sigma_c}(v)= \frac{(1+v)^{n}}{1-q}\;
\sum_{k\geq 0} \; (q^{k+1})_{k+1} \;
 \left(\sum_{\max
k_i=k} C_{L,\fk}(v)\,\prod\limits^{s_+}_{i=1}
(-v)^{-k_i}\times\right. \ee
$$\left. B_{k_i}((-1)^{c_i+k_i+1} I, v)
\prod\limits^{n}_{i=s_+ +1} B_{k_i}((-1)^{c_i+1} I,v)
 \prod\limits^{n+l_+}_{i=n+1}
(-1)^{k_i} q^{-\frac{k_i(k_i+3)}{2}}\right)
 \,
$$

More generally, $M'=M\#\BR P^3$ can always be obtained by surgery
on an algebraically split link. We set
$I_{M,\sigma}:=I_{M',\sigma'}+I_{M',\sigma''}$
where $\sigma'|_{M}=\sigma''|_M=\sigma$.

\begin{Thm} \label{main-coho} For $M\in \cM_n$,
  a cohomological class  $\sigma$ on $M$,
 there exists a unique
 invariant $I_{M,\sigma}(v)\in \G^{S_2}$,
such that
 for
any $4k$--th root of unity $\zeta$ with odd $k$ and $\zeta^{k^2}=(-1)^{\chi}
 I$,  we have
 $\ev_{\xi}(I_{M,\sigma}(v)) = \tau'_{M,{\sigma+\chi h}}(\xi)$.
\end{Thm}


\begin{proof}
We define the refined Laplace transforms as follows:
$$L^\e_{\pm 2;\l}(q^{a\l})=\frac{1\mp (-1)^{a+\e}\chi I}{2}\;  v^{\mp a^2}$$
Then  we have
$$\g_{\pm2}(\xi)\, \ev_\xi \left(L^\e_{\pm 2;\l}(q^{a\l})\right) =
 {\sum_{\l}}^{\xi,\e} q^{\pm\frac{(\l^2-1)}{2}}
q^{a\l}\, .$$
By shifting $\l\to \l+r/2$, we see that $\gamma^1_{\pm 2,r}=
\pm \chi I\gamma^0_{\pm2,r}$.
This allows us to express the refined Laplace transforms
through the non--refined one.
$$L_{\pm 2}^1(q^{a\lambda }) - L_{\pm 2}^0(q^{a\lambda }) = \pm \chi
 IL_{\pm 2}(q^{a\lambda})|_{v\to -v}$$
$$
L_{\pm 2}^\e=\frac{1}{2}(L_{\pm 2} \pm (-1)^{\e+1}\chi I L_{\pm 2}|_{v\to-v})$$

The rest is analogous to the previous subsection.
\end{proof}

\bibliographystyle{amsplain}

\end{document}